\documentclass{birkau}

\usepackage{amsmath,amssymb,url}
\usepackage{enumitem}
\usepackage{amsfonts, amssymb, amsthm, amsmath, calc, cancel, cite,color, eucal, fullpage, graphics, graphicx, latexsym, mathdots, multirow, pgfplots, theoremref, tikz,tikz-cd, url}
\usepackage{accents}
\usepackage{lineno}
\usepackage{array}
\usepackage{caption}

\numberwithin{equation}{section}

\theoremstyle{plain}
\newtheorem{theorem}{Theorem}[section]
\newtheorem{lemma}[theorem]{Lemma}
\newtheorem{proposition}[theorem]{Proposition}
\newtheorem{corollary}[theorem]{Corollary}

\theoremstyle{definition}
\newtheorem{definition}[theorem]{Definition}

\newtheorem{example}[theorem]{Example}

\DeclareMathOperator{\id}{Id}
\renewcommand{\phi}{\varphi}
\DeclareMathOperator{\con}{Con}
\DeclareMathOperator{\supp}{supp}
\newcommand{\mb}[1]{\mathbf{#1}}
\newcommand{\bb}[1]{\mathbb{#1}}
\newcommand{\mrm}[1]{\mathrm{#1}}
\newcommand{\mc}[1]{\mathcal{#1}}
\newcommand{\A}{\mb{A}}
\newcommand{\D}{\mb{D}}

\newcommand{\V}{\mb{V}}

\newcommand{\X}{\mrm{X}}
\newcommand{\Y}{\mb{Y}}
\newcommand{\0}{\mb{0}}

\newcommand{\2}{\mb{2}}
\newcommand{\3}{\mb{3}}
\newcommand{\n}{\mb{n}}
\newcommand{\bdot}{\boldsymbol{\cdot}}
\renewcommand{\emptyset}{\varnothing}
\newcommand{\meet}{\wedge}
\newcommand{\join}{\vee}

\begin{document}

\title[Every complete atomic Boolean algebra is the ideal lattice of a cBCK-algebra]{Every complete atomic Boolean algebra is the ideal lattice of a cBCK-algebra}

\author[C. M. Evans]{C. Matthew Evans}
\address{Mathematics Department\\
Washington \& Jefferson College\\Washington, PA 15301\\USA}
\urladdr{https://sites.google.com/view/mattevans}
\email{mevans@washjeff.edu}



\begin{abstract}
Given a complete atomic Boolean algebra, we show there is a commutative BCK-algebra whose ideal lattice is that Boolean algebra. This result is shown to exist within a larger framework involving BCK-algebras of functions, whose ideals and prime ideals are analyzed by way of a specific Galois connection. As a corollary of the main theorem, we show that every discrete topological space is the prime spectrum of a cBCK-algebra.
\end{abstract}

\maketitle

\section{Introduction}\label{sec:intro}

The class of BCK-algebras was introduced in 1966 by Imai and Is\'{e}ki \cite{II66} as the algebraic semantics for a non-classical logic having only implication. While this class of algebras is not a variety \cite{wronski83}, many subclasses are; for example, the subclass of commutative BCK-algebras forms a variety. 

As with any algebra of logic, the ideal theory plays an important role in their study. The ideals of a commutative BCK-algebra correspond to its congruence relations: for a commutative BCK-algebra $\A$, there is a lattice isomorphism $\con(\A)\cong\id(\A)$, where $\con(\A)$ is the lattice of congruences on $\A$ and $\id(\A)$ is the lattice of ideals of $\A$ \cite{yutani77}. Because of this, all results regarding ideals of commutative BCK-algebras could equivalently be stated about congruences; we will use the language of ideals. 

For any BCK-algebra $\A$, commutative or not, the lattice $\id(\A)$ is distributive \cite{palasinski81(2)}. In \cite{evans22}, the author proved that each of the following five types of distributive lattice is the ideal lattice of a commutative BCK-algebra:
\begin{enumerate}
\item any finite chain,
\item any countably infinite chain isomorphic to $\bb{Z}_{\leq 0}\cup\{-\infty\}$,
\item any finite subdirectly irreducible distributive p-algebra,
\item any finite Boolean algebra,
\item any distributive lattice $\D$ whose poset of meet-irreducibles is poset-isomorphic to the order dual $T^\partial$ of some finite rooted tree $T$.
\end{enumerate} Additionally, any finite product of lattices where each factor lies in one of the above classes is the ideal lattice of a commutative BCK-algebra. For example, every divisor lattice is the ideal lattice of a commutative BCK-algebra since it is a finite product of finite chains. We note also that $\id(\A)$ is Boolean for any finite cBCK-algebra $\A$ \cite{rt80}.

Here we extend (4) above by showing that every complete atomic Boolean algebra is the ideal lattice of a commutative BCK-algebra. We present two proofs of this fact. The first proof is presented in Section \ref{2}; it is straightforward but perhaps ad hoc, so we situate this result within a broader context in Section \ref{3}. We define an infinite family of algebras we believe to be of interest in its own right, and which contains the algebra used in the first proof. Some time is spent analyzing the ideals of these algebras, and from this analysis a second proof of the main theorem is derived. In the final section, we present an application of the main theorem: any discrete topological space is the spectrum of a commutative BCK-algebra.


\begin{definition} A \textit{commutative BCK-algebra} (\textit{cBCK-algebra}) is an algebra $\A=\langle A; \bdot, 0\rangle$ of type $(2,0)$ such that 
\begin{enumerate}
\item[]\hspace{-1cm} (BCK1)\; $(x\bdot y)\bdot z = (x\bdot z)\bdot y$
\item[]\hspace{-1cm} (BCK2)\; $x\bdot(x\bdot y)=y\bdot(y\bdot x)$
\item[]\hspace{-1cm} (BCK3)\; $x\bdot x=0$
\item[]\hspace{-1cm} (BCK4)\; $x\bdot 0=x$
\end{enumerate}
for all $x,y,z\in A$.
\end{definition}

As mentioned in the introduction, these algebras are the algebraic semantics of a non-classical logic having only implication. The translation happens by reading the product ``$x\bdot y$'' as ``$y\Rightarrow x$'' and 0 as ``true.''

On any cBCK-algebra $\A$ we define a partial order by: $x\leq y$ if and only if $x\bdot y=0$. It can be shown that $0\bdot x=0$ for all $x\in A$, so 0 is the smallest element of $\A$ with respect to $\leq$. The term operation $x\meet y:=y\bdot(y\bdot x)$ is the greatest lower bound of $x$ and $y$, and $\A$ is a semilattice with respect to $\meet$. In particular, from (BCK2) we have $x\meet y=y\meet x$, and these algebras are called \textit{commutative} for this reason. We also have $x\bdot y\leq x$ with equality if and only if $x\meet y=0$.

For proofs of these, as well as other elementary properties of cBCK-algebras, we point the reader to \cite{it76 ,it78,mj94, rt80, tanaka75, traczyk79, yutani77}. All cBCK-algebras are assumed to be non-trivial.

\begin{definition} An \textit{ideal} of a BCK-algebra $\A$ is a subset $I\subseteq A$ such that
\begin{enumerate}
\item $0\in I$
\item $x\bdot y\in I$ and $y\in I$ implies $x\in I$.
\end{enumerate}
\end{definition}

From this we see that ideals are non-empty downsets. BCK-ideals need not be semilattice-ideals: for example, $\A$ itself always satisfies the above definition, but need not be directed upward. More generally, when the underlying poset of $\A$ is a lattice, every BCK-ideal is a lattice-ideal, but the converse is not true: consider the three-element chain $\3=\{0<1<2\}$ with operation $x\bdot y=\max\{x-y,0\}$. Then $I=\{0,1\}$ is a lattice-ideal since it is a downset and closed under join, but it is not a BCK-ideal since $2\bdot 1=1\in I$ but $2\notin I$.

Denote the set of ideals of $\A$ by $\id(\A)$. This set is a complete lattice with respect to $\cap$, and is known to be distributive \cite{palasinski81(2)}. As mentioned earlier, there is a lattice-isomorphism $\id(\A)\cong \con(\A)$, where $\con(\A)$ is the lattice of congruence relations on $\A$; see \cite{AT77} or \cite{yutani77} for a proof. Of course $\{0\}$ and $\A$ are always ideals, and we will say $\A$ is \textit{simple} if these are the only ideals. 

For $S\subseteq A$, the smallest ideal of $\A$ containing $S$ is the \textit{ideal generated by $S$}, which we denote by $(S]$. The join operation in $\id(\A)$ is given by $I\join J=(I\cup J]$.

If $S=\{a\}$, we will write $(a]$ rather than $(\{a\}]$. In \cite{it76}, Is\'{e}ki and Tanaka showed that $x\in(S]$ if and only if there exist  $s_1, \ldots, s_n\in S$ such that 
\begin{equation*}\label{ideal}
\bigl(\cdots\bigl((x\bdot s_1)\bdot s_2\bigr)\bdot \cdots \bdot s_{n-1}\bigr)\bdot s_n=0\,.\tag{$\ast$}
\end{equation*}
This characterization is true in \textit{any} BCK-algebra, commutative or not.

For $n\in\bb{N}_0$, define the notation $x\bdot y^n$ recurvisely as
\begin{align*}
x\bdot y^0&=x\\
x\bdot y^n&=(x\bdot y^{n-1})\bdot y\,.
\end{align*} This gives a decreasing sequence
\[x\bdot y^0\geq x\bdot y^1\geq x\bdot y^2\geq \cdots \geq x\bdot y^n\geq\cdots\,.\] With this notation, we can describe principal ideals by
\begin{equation*}\label{principal ideal}
(a]=\{x\in\A\,\mid\, x\bdot a^n=0\text{ for some $n\in\bb{N}$}\}\,.\tag{$\ast\ast$}
\end{equation*}

If the underlying poset of a cBCK-algebra is totally ordered, we will call it a \textit{cBCK-chain}. The following statement appears as a proposition in \cite{rt82}:

\begin{quotation}
A cBCK-chain $\A$ is simple if and only if, for any $x,y\in\A$ with $y\neq 0$, there exists $n\in\bb{N}$ such that $x\bdot y^n=0$
\end{quotation}
but it is stated without proof. The present author gave a proof in \cite{evans22}, but that proof does not appear to use the fact that $\A$ is commutative or a chain. We provide the proof here:

\begin{proposition}\label{characterization of simple algebras}
A BCK-algebra $\A$ is simple if and only if, for any $x,y\in \A$ with $y\neq 0$, there is a natural number $n$ such that $x\bdot y^n=0$.
\end{proposition}

\begin{proof} Assume first that $\A$ is simple, and take $x,y\in \A$ with $y\neq 0$. Consider the ideal $(y]$. By simplicity we must have $(y]=\A$ since $y\neq 0$. But this means $x\in(y]$, in which case (\ref{principal ideal}) indicates there exists $n\in\bb{N}$ such that $x\bdot y^n=0$. 

On the other hand, assume for any pair $x,y\in\A$ with $y\neq 0$ that there exists $n\in\bb{N}$ such that $x\bdot y^n=0$. Let $I$ be a non-zero ideal of $\A$. Take $z\in \A$ and $y\neq 0$ in $I$. By hypothesis there is some $k\in\bb{N}$ such that $z\bdot y^k=0\in I$. Since $y\in I$, we apply the ideal property inductively to obtain $z\in I$. Hence, $I=\A$, which implies $\A$ is simple. 
\end{proof}

For example, the set of non-negative reals $\bb{R}^+$ is a simple cBCK-chain under the operation $x\bdot y=\max\{x-y, 0\}$. More generally, the positive cone of any Archimedean group can be viewed as a cBCK-algebra, and with this structure it is a simple cBCK-chain. The algebra $\bb{R}^+$ has many important subalgebras -- the unit interval $[0,1]$, the finite chains $\n=\{0<1<\cdots <n\}$, the natural numbers $\bb{N}_0$ -- all of which are simple. For readers familiar with the language of universal algebra: the variety of cBCK-algebras is generated by the finite chains. That is, $\mathtt{cBCK}=HSP(\2, \3, \ldots )$. However, not all BCK-chains are simple; see Example \ref{ordinal sum example} below.

On the other hand, there are many simple BCK-algebras that are not chains. For example, the algebra $\mb{Y}$ defined by Table \ref{tab:Y} is simple, but not a chain. Figure \ref{fig:Y} shows the Hasse diagram.
\begin{figure}[h]
\begin{tabular}{*{2}{>{\centering\arraybackslash}b{\dimexpr0.5\linewidth-2\tabcolsep\relax}}}
\begin{tabular}{c||cccc}
$\bdot$   & 0 & 1 & 2 & 3\\\hline\hline
0             & 0 & 0 & 0 & 0\\
1             & 1 & 0 & 0 & 0\\
2             & 2 & 1 & 0 & 1\\
3             & 3 & 1 & 1 & 0\\        
\end{tabular}
\captionof{table}{\label{tab:Y} The algebra $\Y$}
&
\begin{tikzpicture}
\filldraw (0,0) circle (2pt);
\filldraw (0,1) circle (2pt);
\filldraw (1,2) circle (2pt);
\filldraw (-1,2) circle (2pt);
\draw [-] (0,0) -- (0,1);
\draw [-] (0,1) -- (-1,2);
\draw [-] (0,1) -- (1,2);
	\node at (0,-.4) {\small 0};
	\node at (.3,.9) {\small $1$};
	\node at (-1, 2.3) {\small $2$};
	\node at (1, 2.3) {\small $3$};
\end{tikzpicture}
\caption{\label{fig:Y}}
\end{tabular}
\end{figure}


\section{The Main Theorem}\label{2}

\begin{theorem}\label{main theorem}
Every complete atomic Boolean algebra is the ideal lattice of a commutative BCK-algebra.
\end{theorem}

\begin{proof}
Let $\mb{C}$ be a complete atomic Boolean algebra. Then $\mb{C}$ is isomorphic to $\mc{P}(X)$, the power set of $X$, for some set $X$. With respect to set difference, $\mc{P}(X)$ is a cBCK-algebra with $\emptyset$ as the zero element. Put $\A=\mc{P}_{\text{fin}}(X)$, the set of finite subsets of $X$, and note that $\A$ is a subalgebra of $\mc{P}(X)$. 

The underlying poset of $\A$ is a lattice, so each BCK-ideal of $\A$ is a lattice-ideal. In this case, every lattice-ideal is also a BCK-ideal; to see this, let $I$ be a lattice-ideal of $\A$. Certainly $\emptyset\in I$. Now suppose $B\setminus C\in I$ and $C\in I$ for some $B,C\in\A$. Since $I$ is a lattice-ideal it is a downset, so $B\cap C\in I$. Therefore $(B\setminus C)\cup (B\cap C)=B\in I$, which tells us $I$ is a BCK-ideal. So BCK-ideals and lattice-ideals coincide in $\A$. 

For any $S\subseteq X$, we note that $\mc{P}_{\text{fin}}(S)$, the set of finite subsets of $S$, is a lattice-ideal (and hence a BCK-ideal) of $\A$ since $\mc{P}_{\text{fin}}(S)$ is closed under $\cup$ and closed downward with respect to $\subseteq$. Thus, we define a map $\phi\colon \mc{P}(X)\to \id(\A)$ by 
\[\phi(S)=\mc{P}_{\text{fin}}(S):=I_S\,.\] We claim this is an isomorphism of Boolean algebras. 

First, it is clear that this map is injective. To see it is surjective, take $I\in\id(\A)$. Putting $S=\bigcup_{B\in I}B$, we have $I=I_S$; this follows since every element $B\in I$ is a finite union of elements of $S$. Hence, $\phi$ is bijective.

Next we show $\phi$ is a lattice homomorphism. For $S,T\subseteq X$, we have $I_{S\cap T}=I_S\cap I_T$ just by general properties of the powerset. To show $\phi$ preserves joins, take $B\in I_{S\cup T}$. Then we can write $B$ as a finite union of elements of $S\cup T$, which can be rearranged to get $B=C\cup D$ where $C$ is a finite subset of $S$ and $D$ is a finite subset of $T$. This implies $B\in I_S\join I_T$, so $I_{S\cup T}\subseteq I_S\join I_T$. The reverse inclusion is straightforward, and so $I_{S\cup T}=I_S\join I_T$.

Given $I_S$, observe that 
\[I_S\cap I_{S^c}=I_{S\cap S^c}=I_\emptyset=\{\emptyset\}\]
and
\[I_S\join I_{S^c}=I_{S\cup S^c}=I_X=\A\,,\] meaning that $I_{S^c}$ is the complement of $I_S$ in $\id(\A)$. So $\phi$ preserves complements, meaning $\phi$ is an isomorphism of Boolean algebras, as claimed.

Thus, the theorem follows since we have $\mb{C}\cong\mc{P}(X)\cong \id(\A)$.

\end{proof}

\section{BCK-algebras of functions}\label{3}

In this section we show how the above theorem is a special case of a more general result. Let $X$ be any set and $\A$ a non-trivial cBCK-algebra. Put \[\mc{F}(X, \A)=\{f\colon X\to \A\}\,,\] the set of all functions from $X$ to $\A$. When $X$ and $\A$ are understood, we will write $\mc{F}$ instead of $\mc{F}(X,\A)$. The set $\mc{F}$ becomes a cBCK-algebra with pointwise operation $(f\bdot g)(x)=f(x)\bdot_\A g(x)$, where the zero element is the zero function $\0$. This algebra is a generalization of Example 1 from \cite{it78}, and a version also appears as Example 7 of the author's PhD thesis \cite{evans20}. We will use $\leq$ to denote the order on both $\mc{F}$ and $\A$, and it will always be clear from context. We note that the semilattice operation $\meet$ on $\mc{F}$ is also given pointwise; that is, $(f\meet g)(x)=f(x)\meet_\A g(x)$.

For $f\in\mc{F}$, define the \textit{support} of $f$ to be 
\[\supp(f)=\{\, x\in X\,\mid\, f(x)\neq 0\}=\{\, x\in X\,\mid\, f(x)> 0\}\,.\]
We say that $f$ has \textit{finite support} if the support of $f$ is a finite set. Let $\mc{FS}(X,\A)$ be the set of functions $f\colon X\to \A$ with finite support; when $X$ and $\A$ are understood we will write $\mc{FS}$.

\begin{lemma}
$\mc{FS}$ is an ideal of $\mc{F}$, and hence a subalgebra of $\mc{F}$.
\end{lemma}

\begin{proof}
Clearly $\0\in\mc{FS}$. Now suppose $f\notin \mc{FS}$ and $g\in\mc{FS}$. Given $x\in\supp(f)\setminus\supp(g)$, we have 
\[(f\bdot g)(x)=f(x)\bdot_\A g(x)=f(x)\bdot_\A 0=f(x)>0\] which implies that $\supp(f\bdot g)$ is infinite. Thus, $f\bdot g\notin \mc{FS}$, and we have that $\mc{FS}$ is an ideal of $\mc{F}$.

Ideals are always subalgebras, so the result follows.
\end{proof}

Consider the relation $R\subseteq \mc{F}\times X$ defined by 
\[R=\{\, (f,x)\in \mc{F}\times X\,\mid\, f(x) = 0\,\}\,.\] This relation induces a Galois connection:
\begin{itemize}\itemsep=2ex
\item for $\mc{G}\subseteq \mc{F}$, put 
	\begin{align*}
		P(\mc{G})&=\{\, x\in X\,\mid\, g(x)=0 \text{ for all $g\in \mc{G}$}\,\}\\
		&= \bigcap_{g\in\mc{G}}\bigl(X\setminus\supp(g)\bigr)
	\end{align*}
\item for $S\subseteq X$, put 
	\begin{align*}
		V(S)&=\{\, f\in\mc{F}\,\mid\, f(s)=0\text{ for all $s\in S$}\,\}\\
		&=\{\,f\in\mc{F}\,\mid\, S\subseteq X\setminus\supp(f)\,\}\,.
	\end{align*}
\end{itemize} For a singleton $\{x\}\subseteq X$, we will write $V(x)$ for $V(\{x\})$. We note this is an antitone Galois connection: if $S\subseteq T\subseteq X$, then certainly $V(T)\subseteq V(S)$ since any function vanishing on $T$ must also vanish on $S$. Similarly, $P(-)$ is order-reversing.

\begin{lemma}\label{V(S) is an ideal}
The set $V(S)$ is an ideal of $\mc{F}$ for any $S\subseteq X$.
\end{lemma}

\begin{proof}
That $\0\in V(S)$ is clear. Suppose $f\bdot g\in V(S)$ and $g\in V(S)$. If we had $f(s)> 0$ for some $s\in S$, then $(f\bdot g)(s)> 0$ since $g(s)=0$, but this contradicts the fact that $f\bdot g\in V(S)$. So we must have $f(s)=0$ for all $s\in S$, and thus $f\in V(S)$. Hence, $V(S)$ is an ideal of $\mc{F}$.
\end{proof}

\begin{lemma}\label{V is injective}
The map $V\colon \mc{P}(X)\to\id(\mc{F})$ is injective.
\end{lemma}

\begin{proof}
Let $S,T\subseteq X$ with $S\neq T$. Without loss of generality, there is some $s\in S\setminus T$. Pick some non-zero $a\in\A$, and define $f\colon X\to\A$ by $f(x)=0$ for all $x\in X\setminus\{s\}$ and $f(s)=a$. Then $f\in V(T)$ but $f\notin V(S)$, and hence $V(S)\neq V(T)$.
\end{proof}

However, $V$ need not be surjective. For example, consider \[\mc{F}(\bb{N},\2)=\2^\bb{N}\cong \mc{P}(\bb{N})\,.\] Notice that $\mc{FS}(\bb{N},\2)\cong \mc{P}_{\text{fin}}(\bb{N})$ contains the characteristic function $\chi_{\{n\}}$ for each $n\in\bb{N}$, and so $\mc{FS}\neq V(S)$ for any $S\subseteq\bb{N}$.

On the other hand, if $X$ is a finite set, say $|X|=n$, then $\mc{F}(X,\A)=\mc{FS}(X,\A)\cong \A^n$, and it is known that every ideal of $\A^n$ is a product of ideals of $\A$ (see Lemma 2.1.1 of \cite{evans20}). Thus, if $\A$ is simple, then every ideal $I$ of $\A^n$ is of the form $I=\prod_{i=1}^nE_i$ where each factor $E_i$ is either $\{0\}$ or $\A$, meaning every ideal of $\mc{F}(X,\A)=\mc{FS}(X,\A)$ is of the form $V(S)$ for some $S\subseteq X$. And in fact, we observe that $I=VP(I)$ for any ideal $I$ of $\mc{F}$. That is, ideals are the \textit{closed sets} of the Galois connection $(V,P)$ in this case. This can be extended to infinite $X$, provided we restrict to the subalgebra $\mc{FS}(X,\A)$.

\begin{theorem}\label{ideals of FS}
Let $X$ be any set and $\A$ a simple cBCK-algebra. A subset $I\subseteq \mc{FS}(X,\A)$ is an ideal of $\mc{FS}$ if and only if $I=VP(I)$. In particular, every ideal of $\mc{FS}$ has the form $V(S)$ for some $S\subseteq X$.
\end{theorem}

\begin{proof} First, note that the proof of Lemma \ref{V(S) is an ideal} is also true for $\mc{FS}$. Thus, if $I=VP(I)$, we see that $I$ is an ideal.

On the other hand, suppose $I$ is an ideal of $\mc{FS}$. Since $V$ and $P$ are a Galois connection, we have $I\subseteq VP(I)$. For the other inclusion, take $f\in VP(I)$ with $f\neq \0$. We know that $\supp(f)$ is finite, so enumerate the elements $\supp(f)=\{\,x_1, x_2,\ldots, x_k\,\}$. For each $x_i\in\supp(f)$, the fact that $f(x_i)\neq 0$ tells us $x_i\notin PVP(I)=P(I)$. Hence, for each $i=1,\ldots, k,$ there is an element $g_i\in I$ such that $g_i(x_i)\neq 0$. 

Since $\A$ is simple, Proposition \ref{characterization of simple algebras} tells us that for each $i$ there is $n_i\in\bb{N}$ such that \[(f\bdot g_i^{n_i})(x_i)=f(x_i)\bdot_\A \bigl(g_i(x_i)\bigr)^{n_i}=0\,.\]  Define $h:=\bigl(\cdots \bigl((f\bdot g_1^{n_1})\bdot g_2^{n_2}\bigr)\bdot \cdots \bdot g_{k-1}^{n_{k-1}}\bigr)\bdot g_k^{n_k}$. We claim that $h=\0$.

First, it may the be case that $g_i(x_j)\neq 0$ for $i\neq j$. By (BCK1) the order in which the $g_i$'s are applied does not matter, so we have
\[\bigl((f\bdot g_i^{n_i})\bdot g_j^{n_j}\bigr)(x_j)=\bigl((f\bdot g_j^{n_j})\bdot g_i^{n_i}\bigr)(x_j)\leq (f\bdot g_j^{n_j})(x_j)=0\,,\]
which implies that $h(x_j)=0$ for all $x_j\in\supp(f)$.

Next, for any $z\in X\setminus \supp(f)$ we have $f(z)=0$ which implies $h(z)=0$; so $h$ vanishes on $X\setminus\supp(f)$ as well. 

Taking these observations together, we have
\[h=\bigl(\cdots \bigl((f\bdot g_1^{n_1})\bdot g_2^{n_2}\bigr)\bdot \cdots \bdot g_{k-1}^{n_{k-1}}\bigr)\bdot g_k^{n_k}=\0\in I\,.\] Since each $g_i\in I$, we inductively apply the ideal property to see $f\in I$. Thus, $VP(I)\subseteq I$ and we have $I=VP(I)$ as claimed.
\end{proof}

And as the following corollary shows, we can only characterize ideals in this way when $\A$ is simple:

\begin{corollary}\label{closed sets}
The ideals of $\mc{FS}(X,\A)$ coincide with closed sets of the Galois connection $(V,P)$ if and only if $\A$ is simple.
\end{corollary}

\begin{proof}
If $\A$ is simple, we apply Theorem \ref{ideals of FS}.

We prove the converse by contrapositive. Suppose $\A$ is not simple. Let $J$ be a non-zero proper ideal of $\A$ and consider \[I=\{ f\in\mc{FS}\,\mid\, f(x)\in J\text{ for all $x\in X$}\}\,.\] We first show that $I$ is an ideal of $\mc{FS}$.

Clearly $\0\in I$, so suppose $f\bdot g\in I$ and $g\in I$. Take $x\in X$. Then $(f\bdot g)(x)=f(x)\bdot_\A g(x)\in J$ and $g(x)\in J$. Since $J$ is an ideal of $\A$ this implies that $f(x)\in J$, and since $x$ was arbitrary we have $f\in I$. So $I\leq \mc{FS}$.

Now fix $y\in J$ with $y\neq 0$. For each $x\in X$, define the function $f_x$ by 
$f_x(z)=
\left.\begin{cases}
y &\text{if $z=x$}\\
0 &\text{otherwise}
\end{cases}\right\}\,.$ Clearly $f_x\in I$ for all $x\in X$, and so $P(I)=\emptyset$. Hence $I\subsetneq VP(I)=\mc{FS}$. So $I$ is not a closed set.
\end{proof}

\begin{example}\label{example V}
As a small example of this, consider the algebra $\V$ defined by Table \ref{tab:V}.
\begin{figure}[h]
\begin{tabular}{*{2}{>{\centering\arraybackslash}b{\dimexpr0.5\linewidth-2\tabcolsep\relax}}}
\begin{tabular}{c||ccc}
$\bdot$   & 0 & 1 & 2 \\\hline\hline
0             & 0 & 0 & 0 \\
1             & 1 & 0 & 1 \\
2             & 2 & 2 & 0 \\     
\end{tabular}
\captionof{table}{\label{tab:V} The algebra $\V$}
&
\begin{tikzpicture}
\filldraw (0,0) circle (2pt);
\filldraw (-1,1) circle (2pt);
\filldraw (1,1) circle (2pt);
\draw [-] (0,0) -- (-1,1);
\draw [-] (0,0) -- (1,1);
	\node at (0,-.4) {\small 0};
	\node at (-1, 1.3) {\small $1$};
	\node at (1, 1.3) {\small $2$};
\end{tikzpicture}
\caption{\label{fig:V}}
\end{tabular}
\end{figure}
This algebra has two non-zero proper ideals, $J=\{0,1\}$ and $K=\{0,2\}$, which means that 
\[I=J\times\V=\{\,(0,0)\,,\,(0,1)\,,\,(0,2)\,,\,(1,0)\,,\,(1,1)\,,\,(1,2)\,\}\] is an ideal of $\V^2\cong\mc{FS}\bigl(\{1,2\}, \V\bigr)$. Concretely we see that $P(I)=\emptyset$ and $I\subsetneq \V^2 =VP(I)$.
\end{example}

The above corollary also implies the following:

\begin{corollary}\label{V is bijective}
The map $V\colon \mc{P}(X)\to\id(\mc{FS})$ is bijective if and only if $\A$ is simple.
\end{corollary}

\begin{proof}
This follows from Lemma \ref{V is injective} and Corollary \ref{closed sets}.
\end{proof}

\begin{lemma}\label{properties of V}
For \textit{any} cBCK-algebra $\A$, the map $V$ is a lattice anti-homomorphism. That is, for $Y,Z\subseteq X$ we have 
\begin{align*}
V(Y)\cap V(Z)&=V(Y\cup Z) \text{ and }\\
V(Y)\join V(Z)&=V(Y\cap Z)\,.
\end{align*}
\end{lemma}

\begin{proof}
We begin with the first equality. Since $Y,Z\subseteq Y\cup Z$ we have $V(Y\cup Z)\subseteq V(Y), V(Z)$ since $V$ is order-reversing, and hence $V(Y\cup Z)\subseteq V(Y)\cap V(Z)$. For the other inclusion, simply note that any function vanishing on both $Y$ and $Z$ necessarily vanishes on $Y\cup Z$.

For the second equality, begin by noting that $Y\cap Z\subseteq Y,Z$, and so $V(Y), V(Z)\subseteq V(Y\cap Z)$. This implies $V(Y)\join V(Z)\subseteq V(Y\cap Z)$. For the opposite inclusion, take $f\in V(Y\cap Z)$ and let $\supp(f)=\{\, x_1, x_2, \ldots, x_k\,\}\subseteq Y^c\cup Z^c$.  Set $a_i:=f(x_i)$ for each $i$, and define $g_i\colon X\to \A$ by $g_i(x_i)=a_i$ and $g(x)=0_\A$ for all other $x\in X$. Observe that $\{g_1, g_2, \ldots, g_k\}\subseteq V(Y)\cup V(Z)$. Then \[\bigl(\cdots \bigl((f\bdot g_1)\bdot g_2\bigr)\bdot \cdots \bdot g_{k-1}\bigr)\bdot g_k=\0\,,\] and since each $g_i\in V(Y)\cup V(Z)$ we have $f\in \bigl(V(Y)\cup V(Z)\bigr]=V(Y)\join V(Z)$ by \eqref{ideal}. Thus, $V(Y\cap Z)\subseteq V(Y)\join V(Z)$ and the equality follows.
\end{proof}

\begin{theorem}\label{Id(FS) is Boolean}
The map $V\colon\mc{P}(X)\to \id(\mc{FS})$ is an (anti-)isomorphism of Boolean algebras if and only if $\A$ is simple.
\end{theorem}

\begin{proof}
Assume $\A$ is simple. We know $\id(\mc{FS})$ is a bounded distributive lattice (whether $\A$ is simple or not). Let $I$ be an ideal of $\mc{FS}$. Since $\A$ is simple, Theorem \ref{ideals of FS} gives us $I=V(S)$ for some $S\subseteq X$. Then the complement of $I$ in $\id(\mc{FS})$ is the obvious natural candidate $V(S^c)$, which follows from Lemma \ref{properties of V}:
\begin{align*}
V(S)\cap V(S^c)&=V(S\cup S^c)=V(X)=\{\0\}\\
V(S)\join V(S^c)&=V(S\cap S^c)=V(\emptyset)=\mc{FS}\,.
\end{align*} Thus, $\id(\mc{FS})$ is a Boolean algebra.

By Lemma \ref{V is bijective}, since $\A$ is simple we know that $V$ is bijective, and the preceding paragraph shows that $V$ sends complements in $\mc{P}(X)$ to complements in $\id(\mc{FS})$. Lemma \ref{properties of V} shows that $V$ sends meets to joins and vice versa. So $V$ is an Boolean anti-isomorphism, but Boolean algebras are self-dual. Hence, $\id(\mc{FS})\cong \mc{P}(X)$.

Conversely, if $\A$ is not simple then $V$ is not surjective and hence not an isomorphism.
\end{proof}

This now gives us another proof of Theorem \ref{main theorem}: letting $\mb{C}$ be a complete atomic Boolean algebra, we have $\mb{C}\cong \mc{P}(X)$ for some set $X$. The cBCK-algebra $\2$ is simple and so $\id\bigl(\mc{FS}(X,\2)\bigr)\cong\mc{P}(X)$, which gives the result.

This is, of course, the same proof in a different disguise, but we now see that every $\mb{C}$ occurs as the ideal lattice of infinitely many non-isomorphic cBCK-algebras. Further, from Lemma \ref{V is injective} and Lemma \ref{properties of V}, we see that the Boolean algebra $\mc{P}(X)$ is embedded in the lattice $\id\bigl(\mc{F}(X,\A)\bigr)$. We believe these algebras $\mc{F}(X,\A)$ -- and their ideal lattices -- to be an interesting topic for future study.


\section{A small application}\label{4}

In the final section, we present a small application of this theorem by showing that any discrete topological space is the prime spectrum of a cBCK-algebra. We first remind the reader of some definitions.

\begin{definition}
Let $\mb{L}$ be a meet-semilattice. An element $m\in\mb{L}$ is \textit{meet-irreducible} if $x\meet y=m$ implies either $x=m$ or $y=m$.
\end{definition}

\begin{definition}A proper ideal $J$ of a cBCK-algebra is \textit{prime} if $x\meet y\in J$ implies $x\in J$ or $y\in J$.
\end{definition}

It is easy to see, for example, that $\{0\}$ is a prime ideal of $\A$ if and only if $0$ is meet-irreducible in $\A$. Further, Pa\l asinski proved in \cite{palasinski81} that the prime ideals of $\A$ are precisely the meet-irreducible elements of the ideal lattice $\id(\A)$. 

\begin{lemma}\label{0 prime}
If $\{0\}$ is a prime ideal, then $V(x)$ is a prime ideal of $\mc{FS}(X,\A)$ for all $x\in X$.
\end{lemma}

\begin{proof}
Take $x\in X$ and suppose $f\meet g\in V(x)$. Then $(f\meet g)(x)=f(x)\meet_\A g(x) = 0$. Since $\{0\}$ is a prime ideal of $\A$, the element 0 is meet-irreducible in $\A$, so either $f(x)=0$ or $g(x)=0$. Thus, either $f\in V(x)$ or $g\in V(x)$, and we see $V(x)$ is a prime ideal.
\end{proof}

However, $\{0\}$ being a prime ideal does not guarantee that \textit{every} prime ideal of $\mc{FS}(X,\A)$ is of the form $V(x)$. We give an example:

\begin{example}\label{ordinal sum example}
Let $\overline{\bb{N}}_0$ denote the order dual of $\bb{N}_0$, and let $\D=\bb{N}_0\oplus\overline{\bb{N}}_0$, where $\oplus$ denotes ordinal sum. As a poset, $\D$ is a chain
\[0<1<2<3<\cdots < \overline{3}<\overline{2}<\overline{1}<\overline{0}\,.\] This becomes a cBCK-algebra compatible with $<$ if we define the operation as follows: for $a,b\in \bb{N}_0$, 
\begin{align*}
a\bdot b&=\max\{a-b, 0\}\,,\\
\overline{a}\bdot\overline{b}&=\max\{b-a, 0\}\,,\\
a\bdot\overline{b}&=0\,,\\
\overline{a}\bdot b&=\overline{a+b}\,.
\end{align*}
This example seems to have first appeared in \cite{rt81}, appeared again shortly after in \cite{wronski83(2)}, and was then generalized more recently in \cite{evans20}. In any case, the algebra $\D$ is subdirectly irreducible with ideal lattice 
$\id(\D)=\bigl\{\,\{0\bigr\}<\bb{N}_0<\D\,\bigr\}$ isomorphic to the three-element chain $\3$. Then the algebra $\mc{F}\bigl(\{1,2\},\D\bigr)\cong\D^2$ has ideal lattice isomorphic to $\3^2$. The ideal $I:=\bb{N}_0\times \D$ is meet-irreducible in $\id(\D^2)$, and thus a prime ideal of $\D^2$, yet $I\neq V(1)$ and $I\neq V(2)$.

\end{example}

On the other hand, if $\{0\}$ is not a prime ideal of $\A$, then ideals of the form $V(x)$ may not be prime. Consider again the algebra $\V^2$ from Example \ref{example V}. The ideal $V(1)=\{0\}\times\V$ is not prime since $(1,0)\meet(2,0)=(0,0)\in V(1)$, but $(1,0)\notin V(1)$ and $(2,0)\notin V(1)$.

The above examples suggest the following:

\begin{proposition}\label{V(x) is prime}
Let $\A$ be simple. An ideal $J$ of $\mc{FS}(X,\A)$ is prime if and only if $J=V(x)$ for some $x\in X$.
\end{proposition}

\begin{proof}
Since $\A$ is simple, 0 must be meet-irreducible and hence all ideals of the form $V(x)$ are prime ideal by Lemma \ref{0 prime}.

Conversely, suppose $J$ is a prime ideal. Since $J$ is an ideal, we know from Theorem \ref{ideals of FS} that $J=V(S)$ for some subset $S\subseteq X$. If $S=\emptyset$ then $V(S)=\mc{FS}$, a contradiction since prime ideals are proper. If $|S|\geq 2$, take $x,y\in S$ with $x\neq y$. Let $f,g\in\mc{FS}$ be such that $f(x)\neq 0$ but $f(s)=0$ for all $s\in S\setminus \{x\}$, while $g(y)\neq 0$ but $g(s)=0$ for all $s\in S\setminus \{y\}$. Then $(f\meet g)(s)=0$ for all $s\in S$, so $f\meet g\in V(S)$. Yet $f,g\notin V(S)$, meaning $J=V(S)$ is not prime, a contradiction. Hence, we must have $|S|=1$, meaning $J$ is of the form $V(x)$ for some $x\in X$.
\end{proof}

Let $\X(\A)$ denote the set of prime ideals of $\A$. For $S\subseteq\A$, define \[\sigma(S)=\{P\in\X(\A)\mid S\not\subseteq P\}\,.\] The collection $\mc{T}(\A)=\{\sigma(I)\mid I\in\id(\A)\}$ is a topology on $\X(\A)$. There are several proofs of this in the literature, but we point the reader to the paper \cite{ADT93}, in which it is also shown that $\mc{T}(\A)\cong \id(\A)$ as lattices. 

\begin{definition} The space $\bigl(\X(\A)\,,\, \mc{T}(\A)\bigr)$ is the \textit{(prime) spectrum} of $\A$.\end{definition}

Given a topological space $Y$, let $\mc{T}_Y$ denote the lattice of open subsets. We say $Y$ is a \textit{spectral space} if it is homeomorphic to the prime spectrum of a commutative ring. Hochster \cite{hochster69} proved that a space $Y$ is a spectral space if and only if $Y$ is compact, sober, and the compact open subsets form a basis that is a sublattice of $\mc{T}_Y$. A \textit{generalized spectral space} is a space satisfying the latter two of these conditions. Thus, a generalized spectral space which is compact is simply a spectral space. 

Meng and Jun \cite{MJ98} proved that the spectrum of a bounded cBCK-algebra is a spectral space. The present author showed in \cite{evans22} that the spectrum of any cBCK-algebra is a locally compact generalized spectral space, with compactness if and only if the algebra is finitely generated as an ideal. 

In general, given a generalized spectral space $Y$, it is difficult to say whether or not that space is the spectrum of a cBCK-algebra. But it is known that $Y\simeq \X(\A)$ for some cBCK-algebra $\A$ if and only if there is a lattice isomorphism $\mc{T}_Y \cong\id(\A)$; see \cite{evans22} for a proof. 

\begin{proposition}
Let $Y$ be a discrete topological space. Then $Y$ is the spectrum of a cBCK-algebra.
\end{proposition}

\begin{proof}
Since $Y$ is a discrete space, the lattice of open subsets is $\mc{P}(Y)$. By Theorem \ref{Id(FS) is Boolean} we have $\mc{T}_Y\cong \mc{P}(Y)\cong \id\bigl(\mc{FS}(Y,\A)\bigr)$ for any simple cBCK-algebra $\A$, and hence $Y\simeq \X\bigl(\mc{FS}(Y,\A)\bigr)$ for any simple cBCK-algebra $\A$ by the preceding paragraph.
\end{proof}

\end{document}